\documentclass{elsarticle}

\usepackage[utf8]{inputenc}

\usepackage{amssymb}
\usepackage{amsmath}
\usepackage{amsthm}
\usepackage{enumitem}
\usepackage{pgf,tikz}
\usetikzlibrary{arrows}

\numberwithin{equation}{section}

\theoremstyle{plain}
\newtheorem{thm}{Theorem}[section]
\newtheorem{lemma}[thm]{Lemma}
\newtheorem{prop}[thm]{Proposition}
\newtheorem{coroll}[thm]{Corollary}
\newtheorem{claim}[thm]{Claim}
\newtheorem{fact}[thm]{Fact}

\theoremstyle{definition}
\newtheorem{defn}[thm]{Definition}

\newtheorem{remark}[thm]{Remark}
\newtheorem{ex}[thm]{Example}

\usepackage{xspace} 

\newcommand{\R}{\mathbb{R}}
\newcommand{\Z}{\mathbb{Z}}

\DeclareMathOperator{\conv}{Conv}

\def\Conv{{\rm Conv}}

\def\Vis{\rm Vis}

\makeatletter
\def\ps@pprintTitle{%
  \let\@oddhead\@empty
  \let\@evenhead\@empty
  \def\@oddfoot{\reset@font\hfil\thepage\hfil}
  \let\@evenfoot\@oddfoot
}
\makeatother

\title{A geometric proof for the root-independence of the greedoid polynomial of Eulerian branching greedoids}

\author{Lilla T\'othm\'er\'esz}
\address{ELTE Eötvös Loránd University, HUN-REN–ELTE Egerváry Research Group on Combinatorial Optimization, P\'azm\'any P\'eter s\'et\'any 1/C, Budapest, Hungary}
\ead{tmlilla@caesar.elte.hu}

\begin{document}

\begin{abstract}
	We define the root polytope of a regular oriented matroid, and show that the greedoid polynomial of an Eulerian branching greedoid rooted at vertex $v_0$ is equivalent to the $h^*$-polynomial of the root polytope of the dual of the graphic matroid.
	
	As the definition of the root polytope is independent of the vertex $v_0$, this gives a geometric proof for the root-independence of the greedoid polynomial for Eulerian branching greedoids, a fact which was first proved by 
	Swee Hong Chan, Kévin Perrot and Trung Van Pham using sandpile models. We also obtain that the greedoid polynomial does not change if we reverse every edge of an Eulerian digraph.
\end{abstract}


\maketitle

\section{Introduction}

It is a well-known fact that for a connected Eulerian digraph, the number of spanning arborescences rooted at a given vertex $v_0$ is independent of the choice of $v_0$. The branching greedoid of a (connected) Eulerian digraph rooted at $v_0$ is a set system where the maximal sets (bases) are the spanning arborescences rooted at $v_0$. The greedoid polynomial $\lambda_{G,v_0}(x)$ is a polynomial invariant of the greedoid, with $\lambda_{G,v_0}(1)$ equal to the number of spanning arborescences rooted at $v_0$.
Kévin Perrot and Trung Van Pham \cite{PP16} and Swee Hong Chan \cite{SweeHong_parking} showed that for an Eulerian digraph $G$, the polynomial $\lambda_{G,v_0}(x)$ does not depend on the choice of the root $v_0$. 
(Perrot and Pham do not talk about the greedoid polynomial, but show the root-independence of a polynomial for Eulerian digraphs. Chan points out that this polynomial is in fact the greedoid polynomial of the branching greedoid, and furthermore, gives a root-independent definition for the polynomial using the sinkless sandpile model.)
Here we give an alternative, geometric proof for the root-independence of $\lambda_{G,v_0}(x)$, by expressing it using the $h^*$-polynomial of a polytope depending only on $G$, and not on $v_0$. This geometric proof also implies that for the Eulerian digraph $\overleftarrow{G}$ obtained by reversing each edge of $G$, the greedoid polynomial $\lambda_{\overleftarrow{G},v_0}(x)$ agrees with $\lambda_{G,v_0}(x)$. 

Our method is the following: We define the root polytope of regular oriented matroids, and show that for the (oriented) dual $M^*$ of the graphical matroid of $G$, the greedoid polynomial $\lambda_{G,v_0}(x)$ is a simple transformation of the $h^*$-polynomial of the root polytope 
of $M^*$. 
The root polytope of $M^*$ does not depend on the choice of $v_0$, hence we obtain the root-independence of $\lambda_{G,v_0}$. Moreover, the reversal of each edge of $G$ does not change the oriented matroid structure, hence also $\lambda_{G,v_0}(x)=\lambda_{\overleftarrow{G},v_0}(x)$.

More precisely, we define the root polytope of a regular oriented matroid using a totally unimodular representing matrix $A$, but it turns out that the important geometric properties are all independent of the representing matrix. Let us state our main theorem explicitly.
\begin{thm}\label{thm:greedoid_poly_as_h^*}
	Let $G=(V,E)$ be an Eulerian digraph, and $v_0$ a fixed vertex of $G$. Let $\lambda_{G,v_0}$ be the greedoid polynomial of the directed branching greedoid of $G$ with root $v_0$. Let $A$ be any totally unimodular matrix representing the oriented dual of the graphic matroid of $G$, and let $\mathcal{Q}_A$ be the root polytope of $A$. Then 
	$$\lambda_{G,v_0}\left(t\right) = t^{|V|-1} h^*_{\mathcal{Q}_A}\left(t^{-1}\right).$$
\end{thm}

We give two different proofs for Theorem \ref{thm:greedoid_poly_as_h^*}.
One of these is based on the definition of the greedoid polynomial as the $h$-vector of the (abstract) dual complex of the greedoid. We show that the complements of spanning arborescences give a unimodular triangulation of the root polytope of the cographic matroid, and this triangulation is a geometric realization of the abstract dual complex, yielding the relationship between the $h$-polynomial and the $h^*$-polynomial. 

Our other proof uses the definition of the greedoid polynomial as a generating function of greedoid activities. For this proof, we generalize the formula from \cite{fixed_edge_order} to get a formula for the $h^*$-polynomial of root polytopes of co-Eulerian regular oriented matroids. Then we show that the activity definition of the greedoid polynomial is a special case of this formula. This second proof also yields further interesting (though, probably not very practical) definitions for the greedoid polynomial.

We note that in the special case of planar Eulerian digraphs (where the dual matroid is also graphic), the relationship of the $h^*$-polynomial of the dual root polytope and the greedoid polynomial was proved in \cite{semibalanced} by Tamás Kálmán and the current author. This paper is the result of an effort to generalize the result of \cite{semibalanced} to all Eulerian digraphs, which needed the introduction of the root polytope for regular oriented matroids.

\section{Preliminaries} \label{sec:prelim}

\subsection{Directed graphs}

Let $G=(V,E)$ be a directed graph. 
A subset $C^*$ of the edges of $G$ is called a \emph{cut} if there is a partition $V=V_1\sqcup V_2$ such that $C^*$ is the set of edges connecting a vertex of $V_1$ and a vertex of $V_2$ (with any orientation). We say that $V_1$ and $V_2$ are the shores of the cut. If each edge of $C^*$ is directed from $V_1$ to $V_2$, or each edge of $C^*$ is directed from $V_2$ to $V_1$, then we say that the cut $C^*$ is directed.

A subgraph of $G$ is called a \emph{tree}, if the underlying unoriented subgraph is connected and cycle-free. (Hence orientations do not play a role in the definition of a tree.) A tree is a spanning tree if it contains each vertex of $G$. For a spanning tree $T$ and edge $e\in T$, the subgraph $T-e$ has two weak connected components. The cut connecting vertices of the two components is called the \emph{fundamental cut} of $e$ with respect to $T$, and is denoted by $C^*(T,e)$.
For an edge $e\notin T$, $T + e$ has exactly one cycle. This cycle is called the fundamental cycle of $e$ with respect to $T$ and is denoted by $C(T,e)$.

Let $v_0$ be a fixed vertex. A tree $F$ of $G$ is called an \emph{arborescence} rooted at $v_0$ if $v_0\in F$, and each edge $e\in F$ is oriented ``away from $v_0$''. That is, in the fundamental cut $C^*(F,e)$, the edge $e$ has its tail in the shore containing $v_0$. A subgraph is called a spanning arborescence, if it is an arborescence, and a spanning tree.

\subsection{Greedoids}

Greedoids were introduced by Korte and Lov\'asz as a generalization of matroids where the greedy algorithm works.
\begin{defn}[greedoid \cite{KorteLovasz}]
	A set system $\mathcal{F}$ on a finite ground set $E$ is called a greedoid, if it satisfies the following axioms
	\begin{itemize}
		\item[(1)] $\emptyset \in \mathcal{F}$,
		\item[(2)] for all $X\in \mathcal{F} - \{\emptyset\}$ there exists $x \in X$ such that $X-x\in \mathcal{F}$,
		\item[(3)] if $X,Y\in \mathcal{F}$ and $|X|=|Y|+1$, then there exist an $x\in X-Y$ such that $Y\cup x \in \mathcal{F}$.
	\end{itemize}
	Elements of $\mathcal{F}$ are called accessible sets, and maximal accessible sets are called bases.
\end{defn}

For example, matroids are a special class of greedoids, but greedoids are able to express connectivity properties that matroids cannot. It follows from the axioms that bases have the same cardinality $r$, which is called the rank of the greedoid.

Here, we will be interested in the class of directed branching greedoids: 
Let $G=(V,E)$ be a directed graph, and let $v_0$ be an arbitrary vertex of $G$.
The \emph{directed branching greedoid with root $v_0$} has groundset $E$, and its accessible sets are the arborescences rooted at $v_0$. (It is easy to check that this is indeed a greedoid.) 
If $G$ is strongly connected (that will always be the case for us), then the bases are the spanning arborescences rooted at $v_0$.

One important invariant of a greedoid is the greedoid polynomial. It was defined by Björner, Korte and Lovász, and they give two equivalent definitions \cite[Theorem 6.1]{greedoid}. One of them is topological, and uses the (abstract) dual complex of the greedoid. Another definition uses activities. We will give two proofs for our main theorem, each based on these two definitions. Hence we repeat here both definitions.

\subsubsection{The definition of the greedoid polynomial via the dual complex}
The \emph{dual complex} of a greedoid is an abstract simplicial complex. By an abstract simplicial complex, we mean a set system $\mathcal{S}$ where $X\in\mathcal{S}$, $Y\subseteq X$ implies $Y\in\mathcal{S}$. Elements $S\in\mathcal{S}$ are called simplices, and $|S|-1$ is called the dimension of $S$. (Hence the dimension of the empty simplex is $-1$.) 

By definition, the dual complex of a greedoid consists of those subsets of $E$ whose complement contains a basis. This is indeed an abstract simplicial complex, and its maximal simplices are the complements of the bases. Hence in particular, the maximal simplices have equal dimension. Such complexes are called \emph{pure}.

The \emph{$h$-polynomial} of a $d$-dimensional simplicial complex $\mathcal{S}$ is defined as
\begin{equation}
\label{eq:h-vektor}
h(x)=f(x-1),\quad\text{where}\quad f(y)= \sum_{S\in \mathcal{S}}y^{d-\dim S}.
\end{equation}
The \emph{greedoid polynomial} is defined as the $h$-polynomial of the dual complex.

\subsubsection{The definition of the greedoid polynomial via activities} 
There is an alternative definition for the greedoid polynomial, using activities.
For a basis $B\in\mathcal{F}$ of a greedoid, an ordering $B=\{b_1, \dots b_r\}$ is called \emph{feasible} if $\{b_1, \dots b_i\}\in \mathcal{F}$ for each $i=1, \dots r$. The axioms guarantee the existence of at least one feasible ordering for each basis.
Let us fix an ordering of the groundset $E$. Now for any basis $B$ of the greedoid, one can associate the lexicographically minimal feasible ordering.

\begin{defn}[external activity for greedoids \cite{greedoid}]
	Fix an ordering of $E$. For a basis $B$, an element $e\notin B$ is externally active in $B$ if for any $f\in B$ such that $B - f + e \in \mathcal{F}$, the lexicographically minimal feasible ordering of $B$ is lexicographically smaller than the lexicographically minimal feasible ordering of $B - f + e$. We call an element $e\notin B$ externally passive in $B$ if it is not externally active in $B$. 
	
	The external activity of a basis $B$ is the number of externally active elements in $B$, and it is denoted by $e(B)$. 
	The external passivity of a basis $B$ is the number of externally passive elements in $B$, and it is denoted by $\bar{e}(B)$. 
\end{defn}

The greedoid polynomial can also be defined as follows:
\begin{defn}[greedoid polynomial, \cite{greedoid}]
	$$\lambda(t) = \sum_{B \text{ basis}} t^{e(B)}$$
\end{defn}

Let $r$ be the commom cardinality of the bases. Then $\bar{e}(B)=r-e(B)$ for each basis $B$. Hence we can also write
$$\lambda(t) = t^{|E|-r} \sum_{B \text{ basis}} (1/t)^{\bar{e}(B)}$$

For the special case of directed branching greedoids, external activity has a more explicit description, as noted by Swee Hong Chan \cite{SweeHong_parking}. In a different context, this activity notion also appears in \cite{LiPostnikov} under the name external semi-activity, hence we use this name here. 
\begin{defn}[external semi-activity in digraphs, \cite{SweeHong_parking,LiPostnikov}]
	Let $G$ be a digraph with a fixed ordering of the edges. Let $T$ be a spanning tree in $G$. An arc $e\notin T$ is externally semi-active for $T$ if in the fundamental cycle $C(T,e)$ the maximal edge (with respect to the fixed ordering) stands in the same cyclic direction as $e$. An arc $e\notin T$ is externally semi-passive for $T$ if it is not externally semi-active.
\end{defn}
One can check that external semi-activity with respect to a spanning arborescence agrees with the greedoid activity in the case of branching greedoids \cite{SweeHong_parking}.

\subsection{Essentials of regular oriented matroids}\label{ssec:matroid}

In this section, we recall the essentials of regular oriented matroids. We try to keep the introduction to a minimum. For more background, see for example \cite{Book_Oriented_matroids}.

\begin{defn}[totally unimodular matrix]
	A matrix is said to be totally unimodular, if each of its subdeterminants is either $0$, $-1$ or $1$.
\end{defn}

A regular matroid $M$ can be represented by a totally unimodular matrix.
Let $A$ be a totally unimodular matrix, and let its columns be $a_1, \dots a_m \in \mathbb{Z}^n$. The elements of $M$ will be the column indices $\{1, \dots, m\}$. We say that a set of elements $\{{i_1} \dots, {i_j}\}$ is \emph{independent}, if the corresponding column vectors are independent over $\mathbb{Q}$. We say that a set $\{{i_1} \dots, {i_j}\}$ is a \emph{basis}, if it is a maximal independent set (equivalently, the corresponding vectors form a basis of $Im(A)$). A minimal dependent set is called a \emph{circuit}.

\begin{remark}
	Note the (well-known) fact, that for a directed graph, the vertex-edge incidence matrix is totally unimodular. If we take $A$ to be the vertex-edge incidence matrix in the above definition, we get back the graphic matroid of the digraph: elements correspond to edges, bases are spanning trees, and circuits are cycles. 
\end{remark}

A totally unimodular matrix also gives us an orientation of the matroid. This means that we can define a sign pattern on the circuits.
A signed circuit is a pair $(C^+,C^-)$ such that $C=C^+ \sqcup C^-=\{{i_1}, \dots {i_j}\}$ is a circuit, and there is a linear dependence  $\sum_{k=1}^j \lambda_k a_{i_k} = 0$ with $\lambda_k > 0$ for ${i_k}\in C^+$ and $\lambda_k < 0$ for ${i_k}\in C^-$. Note that since a circuit is a minimal dependent set, the linear dependence is uniquely determined up to multiplication with a constant. Hence for a circuit $C$, the subsets $C^+$ and $C^-$ can switch role, but the partition is uniquely determined. This means that we can define the relation of whether two elements of a circuit have the same sign or opposite signs: $k, l\in C$ have the same sign in $C$ if they are either both in $C^+$ or both in $C^-$, and they have opposite signs in $C$ if one of them is in $C^-$ and the other one is in $C^+$.
Note that for directed graphs, this definition gives back the usual definition of ``parallel or opposite orientation with respect to a cycle''.

One can define the cocircuits of the matroid in the following way: $C^*=\{{i_1} \dots, {i_j}\}$ is a cocircuit of the matroid if there is a hyperplane $H\subset \mathbb{R}^n$ such that $a_k\in H$ if $k\notin C^*$, and $a_k\notin H$ if $k\in C^*$, moreover, $C^*$ is minimal with respect to this property. One also obtains an orientation for cocircuits: If the hyperplane is defined as $\{x \mid h(x)=0\}$ for the linear functional $h$, then $i\in (C^*)^+$ if $h(a_i) > 0$ and $i\in (C^*)^-$ if $h(a_i) < 0$.. Again, $(C^*)^+$ and $(C^*)^-$ can switch roles, but it is well defined that two elements have the same sign (if $h(a_k)\cdot h(a_l) > 0$) or opposite signs (if $h(a_k)\cdot h(a_l) < 0$) with respect to a cocircuit.

For digraphs, a cut is called \emph{elementary} if it is minimal with respect to inclusion. 
Elementary cuts are exactly the cocircuits for matroids coming from digraphs. Indeed, for a cut, the functional assigning 1 to vertices on one shore and 0 to vertices on the other gives a hyperplane containing exactly the vectors corresponding to the edges not in the cut. Moreover, one can show that conversely, hyperplanes correspond to complements of cuts. Hence elementary cuts are exactly the minimal sets obtainable as the complement of a hyperplane. The relation of having the same or opposite signs with respect to an elementary cut corresponds to whether two edges point towards the same shore. 

Let us say a few words about the totally unimodular matrix $A$ representing a regular oriented matroid $M$. 
Row elimination steps do
not make a difference in the linear dependencies of the columns, 
hence one can suppose that $A$ has linearly independent rows (say, $r$ of them). Moreover, one can transform $A$ to a form $(I_r, X)$ where $I_r$ is an $r\times r$ identity matrix, and $X$ is an $r\times m-r$ totally unimodular matrix. 

Duality is a very important concept of oriented matroids, that generalizes planar duality. A signed set is a set together with a partition into positive and negative parts: $S=S^+ \sqcup S^-$. One says that two signed sets $S_1=S_1^+ \sqcup S_1^-$ and $S_2=S_2^+ \sqcup S_2^-$ are \emph{orthogonal} if either $S_1\cap S_2 = \emptyset$ or $(S_1^+ \cap S_2^+)\cup (S_1^- \cap S_2^-)$ and $(S_1^+ \cap S_2^-)\cup (S_1^- \cap S_2^+)$ are both nonempty.

It is known, that for each oriented matroid $M$ on groundset $E$, there is a unique dual oriented matroid $M^*$ on the same groundset $E$ such that the signed circuits of $M$ and the signed circuits of $M^*$ are mutually orthogonal. Moreover, $(M^*)^*=M$. 
There are also some other very nice relationships between $M$ and $M^*$ that we will explicitly use, hence let us highlight them here.

\begin{fact}\label{fact:bases_circuits_of_dual}
The signed circuits of $M^*$ are exactly the signed cocircuits of $M$. The signed cocircuits of $M^*$ are exactly the signed circuits of $M$. The bases of $M^*$ are exactly the complements of the bases of $M$.
\end{fact}

If $M$ is a regular oriented matroid, and it is represented by the totally unimodular matrix of the form $(I_r, X)$ where $I_r$ is an $r\times r$ identity matrix, and $X$ is an $r\times m-r$ totally unimodular matrix, then $M^*$ is also a regular (oriented) matroid, and it can be represented by the matrix $(-X^T, I_{m-r})$.

If $M$ is the graphic matroid of a planar digraph, then $M^*$ is the graphic matroid of its planar dual. For a nonplanar graph, $M^*$ is not a graphic matroid anymore, but it is still a regular oriented matroid.

Let us point out an important property of circuits and cocircuits that is special to regular matroids. 
By definition, for a circuit $C$ there is a linear combination $\sum_{i\in C} \lambda_i a_{i}=0$. If the matroid is regular, then there is also a linear combination in which $\lambda_i\in\{-1,1\}$ for each $i\in C$. Indeed, let $A'$ be the submatrix of $A$ consisting of columns corresponding to $C$. We can scale down the $\lambda_i$'s, hence the system of inequalities $A'x=0$, $-1\leq x_i \leq 1$ for each $i\in [m]$ has a solution. Extending a TU matrix with columns containing one $1$ or $-1$ elements and zeros yields another TU matrix. 
Hence the inequalities $A'x=0$, $-1\leq x_i \leq 1$ determine an integer polyhedron, which is not empty, and not only the zero vector. Thus, there will be some integer solutions, one of which will not be the zero vector. This solution gives us a nonzero linear combination $\sum_{i\in C} \lambda_i a_{i}=0$ with $\lambda_i\in\{-1,0,1\}$ for each $i\in C$. But as $C$ is minimally dependent, in fact $\lambda_i\in\{-1,1\}$ for each $i\in C$.

Also, for any cocircuit, $C^*$, there is a linear functional $h$ such that 
$$h(a_i) = \left\{\begin{array}{cl} 
0 & \text{if $i\notin C^*$},  \\
1 & \text{if $i \in (C^*)^+$}, \\
-1 & \text{if $i \in (C^*)^-$}.
\end{array} \right.$$
Indeed, suppose that $A=(I_r, X)$. The cocircuit $C^*$ is a circuit of the matroid $M^*$ that is represented by $(-X^T, I_{m-r})$. Hence there is a vector $\mathbf{\lambda}$ with coordinates in $\{0,1,-1\}$ such that $(-X^T, I_{m-r})\mathbf\lambda = 0$, and $\lambda_i = 1$ for the coordinates corresponding to elements of $(C^*)^+$ and $\lambda_i = -1$ for the coordinates corresponding to elements of $(C^*)^-$. Then, the first $r$ coordinates of $\mathbf\lambda$ give a vector $\mathbf\lambda'$ that defines a linear functional $h$ with the above properties.

\subsubsection{Two technical lemmas on circuits of regular oriented matroids}

\begin{claim}\label{cl:zero_comb_as_sum_of_circuits}
	If $\sum_{i=1}^m \lambda_i a_i = 0$, then there exist signed circuits $C_1, \dots C_t$ and coefficients $\nu_1, \dots, \nu_t$ such that $$\lambda_i = \sum_{j: i\in C^+_j} \nu_j - \sum_{j : i\in C^-_j} \nu_j.$$
\end{claim}
\begin{claim}\label{cl:exist_subcircuit_with_good_orientations}
	If $\sum_{i=1}^m \lambda_i a_i = 0$, then there exists a signed circuit $C$ such that $C^+\subseteq \{i: \lambda_i > 0\}$ and $C^-\subseteq \{i: \lambda_i < 0\}$
\end{claim}
\begin{proof}[Proof of Claims \ref{cl:zero_comb_as_sum_of_circuits} and \ref{cl:exist_subcircuit_with_good_orientations}]
	As the elements $\{i : \lambda_i \neq 0\}$ are dependent, there is a subset $C_1$ of them which is a circuit. 
	By subtracting the linear combination $0=\sum_{i\in C_1^+} a_i -\sum_{i\in C_1^-} a_i$ with the appropriate coefficient $\nu_1$, one can achieve that at least one of the coefficients go to zero, while the sign of the other coefficients remain the same. (We can even achieve $\nu_1 > 0$ by taking $C_1^+=C_1^-$ and $C_1^-=C_1^+$ if necessary.) After the modification, we once again have a zero linear combination. Hence we can continue with the procedure, and it is guarateed to end in at most $|\{i: \lambda_i\neq 0\}|$ steps. At the end, we have the desired sum. Notice, that for the last circuit $C_t$, $C_t^+\subseteq \{i: \lambda_i > 0\}$ and $C_t^-\subseteq \{i: \lambda_i < 0\}$, as we always maintained that the remaining part of the linear combination has this property.
\end{proof}

\subsection{Ehrhart theory and computing $h^*$-polynomials}

Suppose that $P\subset\R^n$ is a $d$-dimensional
polytope with vertices in $\Z^n$ (lattice polytope). 
Its Ehrhart polynomial $\varepsilon_P$ is defined as follows: For $t\in \mathbb{N}$, let $\varepsilon_P(t)=|t\cdot P \cap \mathbb{Z}^n|$. This is known to be a polynomial in $t$ (which then can be extended for arbitrary reals). This polynomial is called the Ehrhart polynomial of $P$, and is also denoted by $\varepsilon_P$.

It is known that the polynomials $C_k(t)=\binom{t+d-k}{d}$ for $k=0, \dots d$ give a basis over $\mathbb{Q}$ of the space of at most degree $d$ polynomials with rational coefficients \cite[Lemma 3.8]{KP_Ehrhart}. Hence one can write the Ehrhart polynomial as

$$\varepsilon_P(t) = \sum_{k=0}^d a_k C_k(t).$$

Then the $h^*$ polynomial of $P$ is defined as $h^*(x)=\sum_{k=0}^d a_k x^k$.
(There is another common way to define the $h^*$-polynomial, namely, as the numerator of the Ehrhart series. That definition is equivalent to the one given above.)

We will prove that the greedoid polynomial of a directed branching greedoid is (a simple transformation of) the $h^*$-polynomial of a certain polytope. We will have two proofs for this fact, that correspond to two strategies to compute the $h^*$-polynomial of a polytope. Hence let us briefly summarize these two ways for computing an $h^*$-polynomial.

Both methods use dissections of the polytope into simplices. For a polytope $P$, a \emph{dissection into simplices} means a set of simplices that are interior disjoint, and their union is the whole polytope. A \emph{triangulation} is a dissection into simplices where we additionally require that the intersection of any two simplices is a common face. The advantage of triangulations is that one can think of them as simplicial complexes, but for computing $h^*$-polynomials, dissections are often just as good as triangulations.

A simplex $\Delta = \conv\{p_0, \dots p_d\}$ is said to be unimodular, if $p_0, \dots p_d \in \mathbb{Z}^n$, and one can get each integer point of the linear span of the vectors $p_1 - p_0,  \dots, p_d-p_0$ as an integer linear combination of these vectors. 

The $h$-polynomial of a triangulation is defined as the $h$-polynomial of the abstract simplicial complex.
It is well-known (see for example \cite{semibalanced}), that if a $d$-dimensional polytope $P$ is triangulated into unimodular simplices, then the $h^*$-polynomial of $P$ has the following relationship to the $h$-polynomial of the triangulation:
\begin{equation}\label{eq:h_and_h^*}
\sum_{i=0}^d h^*_i t^i =t^{d+1} h(1/t).
\end{equation}

Hence one method for computing the $h^*$-polynomial is via the $h$-polynomial of a triangulation. (We note that shellable unimodular dissections also yield the $h^*$-polynomial via the same equation, see for example \cite{semibalanced}.)

Let us recall another method, that uses visibilities from a point. This method was used in \cite{arithm_symedgepoly}, and it was generalized to dissections in \cite{fixed_edge_order}.

For a polytope $P$, a dissection into simplices is a collection of simplices $\Delta_1, \dots, \Delta_s$ such that $P=\Delta_1 \cup \dots \cup \Delta_s$, and the simplices are interior disjoint.
We say that a point is in general position with respect to the dissection 
$\Delta_1, \dots, \Delta_s$ if it is not contained in any facet-defining hyperplane of any of the simplices $\Delta_1, \dots, \Delta_s$. 

For two points ${p},{q}\in \mathbb{R}^n$, let us denote by $[p,q]$ the closed segment connecting them, and let us denote by $({p},{q})$ the relative interior of this segment.

We say that a point ${p}$ of a simplex $\Delta_i$ is \emph{visible} from $q$ if $({p},{q})$ is disjoint from $\Delta_i$. We say that a facet of $\Delta_i$ is visible from ${q}$ if all points of the facet are visible from ${q}$. 
Let $\Vis_{q}(\Delta_i)$ be the set of facets of $\Delta_i$ visible from ${q}$.

Then one has the following formula for the $h^*$-polynomial:
\begin{prop}\cite{arithm_symedgepoly,fixed_edge_order} \label{prop:h^*-vector_from_visible_facets}
Let $\Delta_1, \dots \Delta_t$ be a dissection of the $d$-dimensional lattice polytope $P$ into unimodular simplices, and let ${q}\in P$ be a point in general position with respect to the dissection. 
Then the $h^*$-polynomial $h^*(x)=h^*_d x^d +  \dots + h^*_1 x + h^*_0$ of $P$ can be expressed as $$h^*_i = |\{\Delta_j \mid |Vis_{q}(\Delta_j)| = i\}|.$$
\end{prop}
\begin{remark}
Unimodularity of the simplices is a crucial property of the dissection in the above theorem.
\end{remark}

\section{Root polytopes of regular oriented matroids}

In this section, we introduce root polytopes for regular oriented matroids. 
Let $M$ be a regular oriented matroid on the ground set $\{1, \dots, m\}$, and let $A$ be a totally unimodular matrix representing $M$. Let $a_1, \dots a_m$ be the columns of $A$. 

\begin{defn}[Root polytope of a totally unimodular matrix]
Let $A$ be a totally unimodular matrix with columns $a_1, \dots a_m$. We define the root polytope of $A$ as $\mathcal{Q}_A = \conv\{a_1, \dots a_m\}$.
\end{defn}

We show that the important geometric properties of this polytope are in fact independent of the actual representing matrix $A$, hence it makes sense to talk about the root polytope of a regular oriented matroid. (Orientations will play an important role, though.)

Before showing this, let us remark an important special case: The root polytope of a directed graph (also called edge polytope) is the convex hull of the columns of the vertex-edge incidence matrix. Hence the root polytope of a directed graph is precisely the root polytope of the (oriented) graphical matroid.

\begin{prop}
If $A$ and $A'$ are two totally unimodular matrices representing the regular oriented matroid $M$, then there exists a linear bijection $\varphi:\mathcal{Q}_A \to \mathcal{Q}_{A'}$. 
Moreover, $h^*_{\mathcal{Q}_A}=h^*_{\mathcal{Q}_{A'}}$.
\end{prop}
\begin{proof}
Let $a_1, \dots a_m$ be the columns of $A$ and $a'_1, \dots , a'_m$ be the columns of $A'$, and suppose that $a_i$ and $a'_i$ represent the same element $i$ of $M$ for each $i=1, \dots m$.

By definition, any point $p\in\mathcal{Q}_A$ can be written as $\sum_{i=1}^{m}\lambda_i a_i$ where $0\leq \lambda_i$ and $\sum_{i=1}^m \lambda_i = 1$. Then let us define $\varphi(p)=\sum_{i=1}^m \lambda_i a'_i$. 

We show that even though in general a point $p\in\mathcal{Q}_A$ can be written many ways as a convex combination of the columns, the mapping $\varphi$ is well-defined. Let $p=\sum_{i=1}^m \lambda_i a_i =\sum_{i=1}^m \mu_i a_i$. Then $\sum_{i=1}^m (\lambda_i-\mu_i) a_i = 0$. 
By Claim \ref{cl:zero_comb_as_sum_of_circuits} there are circuits $C_1, \dots C_t$ and coefficients $\nu_1, \dots, \nu_t$ such that $\lambda_i-\mu_i = \sum_{j: i\in C^+_j} \nu_j - \sum_{j : i\in C^-_j} \nu_j$. 
As $A$ and $A'$ represent the same regular oriented matroid, the signed circuits $C_1, \dots, C_t$ are also signed circuits of $A'$. This implies that $\sum_{i=1}^m (\lambda_i-\mu_i) a'_i = 0$, hence indeed, the mapping $\varphi$ is well-defined. By symmetry of the role of $A$ and $A'$, $\varphi$ is also injective and surjective, hence it is a bijection between the points of $\mathcal{Q}_A$ and $\mathcal{Q}_{A'}$.

Moreover, notice that the same proof implies that $\varphi$ is also a linear bijection between $t\cdot\mathcal{Q}_A$ and $t\cdot \mathcal{Q}_{A'}$ for any $t\in \mathbb{R}$. Also, if $t$ is an integer, then $\varphi$ takes lattice points of $t\cdot\mathcal{Q}_A$ into lattice points of $t\cdot \mathcal{Q}_{A'}$, and similarly for $\varphi^{-1}$. Indeed, since $A$ and $A'$ are TU, for both $t\cdot\mathcal{Q}_A$ and $t\cdot \mathcal{Q}_{A'}$, lattice points are exactly the integer linear combinations of the column vectors. This implies that for any integer $t$, the number of lattice points in $t\cdot \mathcal{Q}_A$ and in $t\cdot \mathcal{Q}_{A'}$ is the same. Hence also the $h^*$-polynomials are the same.
\end{proof} 

\begin{remark}
	Note that for $\varphi$ to be well-defined, we needed the \emph{signed} circuits to agree. Indeed, keeping the matroid structure, but changing the orientation of the matroid can alter 
	the $h^*$-polynomial. (This can already be seen for graphs.)
\end{remark}

\begin{defn}
	We call a regular oriented matroid \emph{Eulerian}, if $|(C^*)^+|=|(C^*)^-|$ for each cocircuit $C^*$. The name comes from the fact that for a regular oriented matroid that comes from a directed graph, being Eulerian is equivalent to the digraph being Eulerian.
	We call a regular oriented matroid \emph{co-Eulerian} if for each circuit $C$, we have $|C^+| = |C^-|$.
\end{defn} 
By Fact \ref{fact:bases_circuits_of_dual}, a regular oriented matroid $M$ is co-Eulerian if and only if its dual is Eulerian.
Let us state an important special case.
\begin{claim}\label{cl:Eulerian_dual_is_semibalanced}
	For an Eulerian digraph $G$, the (oriented) dual $M^*$ of the graphic matroid of $G$ is a co-Eulerian regular oriented matroid.
\end{claim}
\begin{proof}
	Graphic matroids are regular, and duals of regular matroids are regular, hence $M^*$ is indeed regular. Moreover, the cocircuits of the graphic matroid of $G$ are the elementary cuts of $G$. Any elementary cut of an Eulerian digraph contains equal number of edges going in the two directions, hence the graphic matroid of $G$ is Eulerian. Thus, $M^*$ is co-Eulerian.
\end{proof}

\begin{prop}\label{prop:matroid_root_polytope_dimension}
Let $r$ be the rank of the regular matroid $M$, and let $A$ be a totally unimodular matrix representing $M$.

Then $\dim(\mathcal{Q}_A)= r$ if $M$ is not co-Eulerian, and $\dim(\mathcal{Q}_A)= r-1$ if $M$ is co-Eulerian.

In the latter case, a set of vectors $\{a_{i_1}, \dots a_{i_j}\}$ is affine independent if and only if it is independent in the matroid.
\end{prop}
\begin{proof}
The dimension $\dim(\mathcal{Q}_A)$ is one less than the maximal cardinality of an affinely independent subset of the vectors $\{a_1, \dots a_r\}$. As linearly independent vectors are also affinely independent, the dimension of $\mathcal{Q}_A$ is at least $r-1$. 
Moreover, $\dim(\mathcal{Q}_A)= r-1$ if and only if each linearly dependent set of vectors $\{a_{i_1}, \dots a_{i_j}\}$ is also affinely dependent. This is equivalent to each circuit being affinely dependent. As each circuit has linear dependence $\sum_{i \in C^+} a_i - \sum_{i \in C^-} a_i = 0$, this is an affine dependence if and only if the circuit has $|C^+| = |C^-|$.
If there is any circuit $C$ with $|C^+| \neq |C^-|$, then the corresponding linear dependence is not an affine dependence. Any circuit $C$ can be extended to a set $S$ of size $r+1$ such that $C$ is the only circuit within $S$. Hence if $|C^+|\neq |C^-|$ then $S$ is an affine independent set of size $r+1$, thus, $\dim(\mathcal{Q}_A)=r$.
\end{proof}

\begin{coroll}
If $A$ is a totally unimodular matrix representing a co-Eulerian regular matroid, then convex hulls of bases give maximal dimensional simplices in $\mathcal{Q}_A$.
\end{coroll}

Let $A$ be a totally unimodular matrix representing a co-Eulerian regular matroid $M$. For an independent set $S$, let us denote $\Delta_S = \conv\{a_i: i\in S\} $, which is a simplex in $\mathcal{Q}_A$, and it is a maximal dimensional simplex if and only if $S$ is a basis.
We will call a set of bases $\mathcal{D}$ of $M$ a \emph{dissecting set of bases} if the simplices $\{\Delta_B \mid B\in\mathcal{D}\}$ dissect $Q_A$. That is, if 
$\mathcal{Q}_A = \bigcup_{B\in \mathcal{D}} \Delta_B $, and the simplices in the union are interior disjoint. 

Note that even though we used a representing matrix $A$ for the definition, we claim that whether a  set of bases forms a dissecting set of bases depends only on the oriented matroid $M$. That is, we claim that if $A$ and $A'$ are totally unimodular matrices representing the oriented matroid $M$, then $\mathcal{D}$ is a dissecting set of bases for $A$ if and only if it is a dissecting set of bases for $A'$. Indeed, the bijection $\varphi$ between the points of $\mathcal{Q}_A$ and $\mathcal{Q}_A'$ has the property that $p\in\Delta_B$ if and only if $\varphi(p)\in \Delta'_B$ (where by $\Delta'_B$ we denote the simplex corresponding to $B$ in $\mathcal{Q}_{A'}$). Hence indeed the dissection property is also independent of the representation.

Note the following property, that will ensure that any such dissection is unimodular.
\begin{prop}
	If $M$ is a regular oriented matroid, then for any basis $B$, $\Delta_B$ is a unimodular simplex.
\end{prop}
\begin{proof}
	Let $B=\{i_1, \dots, i_r\}$. Then $\Delta_{B}=\conv\{a_{i_1}, \dots, a_{i_r}\}$. We need to show that the integer linear combinations of $a_{i_1}-a_{i_r}, \dots, a_{i_{r-1}}-a_{i_r}$ generate each integer lattice point in their linear span. It is enough to show that any integer vector that is a linear combination of the vectors $a_{i_1}, \dots, a_{i_r}$ can be obtained as their integer linear combination. This is a well-known property of columns of a totally unimodular matrix (see for example \cite[Lemma 4.3.4]{Frank_konyv}).
\end{proof}

\section{Embedding the dual complex of an Eulerian directed branching greedoid}

In this section we show that one can embed the dual complex of an Eulerian directed branching greedoid as a triangulation of the root polytope of the dual co-Eulerian matroid. From this, it readily follows that the greedoid polynomial is equivalent to the $h^*$-polynomial of the root polytope of the dual matroid. 

\begin{thm}\label{thm:arborescenced_triangulate}
Let $G=(V,E)$ be an Eulerian digraph, and
let $A$ be any totally unimodular matrix representing the (directed) dual of $G$. Let $v_0$ be an arbitrary fixed vertex of $G$.
Let $$\mathcal{D}=\{B \subset E \mid E-B \text{ is a spanning arborescence of $G$ rooted at $v_0$}\}.$$
Then $\{ \Delta_B \mid B\in \mathcal{D}\}$ is a triangulation of $\mathcal{Q}_A$.
\end{thm}

We prove Theorem \ref{thm:arborescenced_triangulate} through two lemmas.
\begin{lemma}\label{lem:arborescences_dissect}
	Using the notation of Theorem \ref{thm:arborescenced_triangulate}, $\{ \Delta_B \mid B\in \mathcal{D}\}$ is a dissection of $\mathcal{Q}_A$ into simplices.
\end{lemma}
\begin{lemma}\label{lem:arborescences_meet_in_common_face}
	Using the notation of Theorem \ref{thm:arborescenced_triangulate}, any two simplices $\Delta_{B_1},\Delta_{B_2}$, $B_1,B_2\in \mathcal{D}$ meet in a common face.
\end{lemma}

We start with proving Lemma \ref{lem:arborescences_dissect}. We will use as a tool an ordering on the set of all spanning trees of $G$. We define this ordering using a fixed ordering of the edges. We note that a similar ordering was introduced in \cite[Section 5]{semibalanced} using a ribbon structure instead of a fixed ordering.

Let us fix an ordering of $E$: $E=\{e_1, \dots e_m\}$. We assign an edge list $List(T)$ to each spanning tree $T$ of $G$ using a type of burning algorithm \cite{Dhar90}:
At the first moment, we say that $v_0$ is the only burning vertex, and no edge is burning.
At any moment, we take the largest edge among the non-burning edges that has at least one burning endpoint (the orientation of the edge does not matter). We add this edge to $List(T)$ and burn it. It the edge was in $T$, then we also burn its non-burning endpoint.
It is easy to see that eventually $List(T)$ includes all edges of $G$.
(It will not be important for us, but note that if $T$ is an arborescence, then the restriction of the edge list to edges of $T$ is a maximal feasible ordering of the edges of $T$.)

\begin{figure}
	\begin{tikzpicture}[scale=.20]
	\node [] (0) at (-2,10) {$s$};
	\node [circle,fill,scale=.8,draw] (1) at (0,10) {};
	\node [circle,fill,scale=.8,draw] (2) at (4,18) {};
	\node [circle,fill,scale=.8,draw] (3) at (9,14) {};
	\node [circle,fill,scale=.8,draw] (4) at (9,6) {};
	\node [circle,fill,scale=.8,draw] (5) at (4,2) {};		
	\node [circle,fill,scale=.8,draw] (6) at (16,10) {};		
	\path [thick,->,>=stealth] (1) edge [left] node {1} (2);
	\path [thick,dashed,<-,>=stealth] (1) edge [above] node {2} (3);
	\path [thick,->,>=stealth] (1) edge [below] node {3} (4);
	\path [thick,dashed,<-,>=stealth] (1) edge [left] node {4} (5);
	\path [thick,dashed,->,>=stealth] (2) edge [above] node {9} (3);
	\path [thick,<-,>=stealth] (5) edge [below] node {6} (4);
	\path [thick,->,>=stealth] (6) edge [above] node {7} (3);
	\path [thick,<-,>=stealth] (6) edge [below] node {8} (4);
	\path [thick,dashed,<-,>=stealth] (4) edge [left] node {5} (3);
	\end{tikzpicture}
	\label{fig:edge_list}
	\caption{Eulerian digraph with root $s$. The non-dashed arcs form a spanning arborescence rooted at $s$.}
\end{figure}
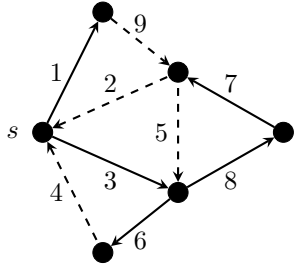
\begin{ex}
	For the rooted spanning arborescence on Figure \ref{fig:edge_list}, the edge-list is $4,3,8,7,9,6,5,2,1$.
\end{ex}

Let us define the labeled edge list of a spanning tree, where we additionally record for each edge whether it is in the tree or not, and if it is in the tree, whether it points towards $v_0$ or away from $v_0$ (within the tree).
We say that the labeled edge-lists of two spanning trees agree up to some point, if until that point, the same edges are listed in the same order, and the status of whether they are in the tree also agrees. 

Let us look at how the first difference between the labeled edge-list of two spanning trees can look like.
Let us take the labeled edge-lists of the trees $T_1$ and $T_2$ at the point before they first differ. Then at this point, the intersection of $T_i$ with the edges of the edge-list is the same for $i=1,2$, hence the set of burning vertices is the same.
Hence we have to choose the next edge of the list from the same set of edges for $T_1$ and for $T_2$ (namely, among the non-burnt edges incident to the set of burning vertices), and by the same rule (we choose the largest edge). Hence the next edge is chosen the same for the two trees. As the two lists start to differ at this point, it can only be because the next edge is in one of the trees, but not in the other.

Based on this observation, we define the following ordering of spanning trees of $G$.
\begin{defn}
	For any two directed spanning trees $T_1$ and $T_2$, we say that $T_1 \prec T_2$ if their labeled edge-lists first differ at an edge $e$ such that either 
	\begin{align}\label{eq:to-el}
	\text{$e\in T_1$, $e\notin T_2$ and $e$ is directed away from $v_0$ in $T_1$,} & \quad\text{ or} \\
	\label{eq:fej-el}
	\text{$e\in T_2$, $e\notin T_1$ and $e$ is directed towards $v_0$ in $T_2$.}&
	\end{align}
\end{defn}
Now we are ready to prove Lemma \ref{lem:arborescences_dissect}. Our proof mimics the proof of \cite[Theorem 5.8]{semibalanced} about the fact that Jaeger trees dissect the root polytope of a semi-balanced directed graph.

\begin{proof}[Proof of Lemma \ref{lem:arborescences_dissect}]
	By Claim \ref{cl:Eulerian_dual_is_semibalanced}, the dual $M^*$ of $G$ is a co-Eulerian regular matroid, hence the maximal affine independent vertex sets in $\mathcal{Q}_A$ correspond to bases of $M^*$. On the other hand, bases of $M^*$ are exactly the complements of (directed) spanning trees of $G$. Spanning arborescences are a subclass of directed spanning trees, hence $\{ \Delta_B \mid B\in \mathcal{D}\}$ are simplices of maximal dimension.
	
	We need to prove that for $B_1, B_2 \in \mathcal{D}$, the simplices $\Delta_{B_1}$ and $\Delta_{B_2}$ are interior-disjoint.
	Moreover, we need to show that for any ${p}\in \mathcal{Q}_A$, there is some $B\in \mathcal{D}$ such that ${p}\in\Delta_B$.
	
	First, we show that for any ${p}\in \mathcal{Q}_A$, there is some $B\in \mathcal{D}$ such that ${p}\in\Delta_B$. By definition, any ${p}\in \mathcal{Q}_A$ is in the convex hull of $\{a_1, \dots a_m\}$. By Caratheodory's theorem, there is a maximal affine independent subset $\{a_{i_1}, \dots a_{i_r}\}$ such that ${p}$ is in their convex hull. Then $\{{i_1}, \dots, {i_r}\}=:B$ is a basis of $M^*$, hence its complement $E-B$ is a directed spanning tree of $G$. To sum up, ${p}\in\Delta_B$ where $E-B$ is a directed spanning tree. If $E-B$ is an arborescence rooted at $v_0$, then we are ready. 
	
	We show that if
	${p}\in \Delta_B$ such that $T=E-B$ is not an arborescence rooted at $v_0$, then there is another basis $B'$ of $M^*$ such that for $T'=E-B'$ we have $T' \prec T$, and ${p}\in \Delta_{B'}$.
	
	Let $p = \sum_{i=1}^m \lambda_i a_i$, where $\lambda_i = 0$ if $i\notin B$ and $\lambda_i\geq 0$, $\sum_{i=1}^m \lambda_i = 1$.
	Take the edge list of $T$. If $T$ is not an arborescence rooted at $v_0$, then there is at least one edge in $T$ that is oriented towards $v_0$. Let $e$ be the first edge in the edge list of $T$ with this property. Take the fundamental cut $C^*(T,e)$. The cut $C^*(T,e)$ is a circuit in the dual matroid $M^*$, hence there is a linear combination $$\sum_{i\in (C^*(T,e))^+} a_i-\sum_{i\in (C^*(T,e))^-} a_i = 0.$$ By symmetry, we can suppose that $e\in C^*(T,e)^+$. 
	Take 
	$$\lambda'_i = \left\{\begin{array}{cl} 
	\lambda_i & \text{if $i\notin C^*(T,e)$},  \\
	\lambda_i +\varepsilon & \text{if $i \in (C^*(T,e))^+$}, \\
	\lambda_i -\varepsilon & \text{if $i \in (C^(T,e)*)^-$}.
	\end{array} \right.$$
	Then $p=\sum_{i=1}^m \lambda'_i a_i$. Take $\varepsilon = \min\{\lambda_i : i \in (C*(T,e))^-\}$. Then the coefficients $\lambda'_i$ are nonnegative. 
	The vectors where $\lambda'_i > 0$ are a subset of $B+e = E-T + e$, moreover, the coefficient of $e$ became positive, and the coefficient of at least one element from $(C^*(T,e))^-$ became zero. Let $f\in  (C^*(T,e))^-$ be such an element. Then since $f$ is in $C^*(T,e)$, $T'=T-e+f$ is a spanning tree. Also, by the above remark, $p\in \Delta_{B+e-f}$. 
	
	We claim that $T'\prec T$. 
	There are two cases. Either $f$ is burnt before $e$ when building up the edge list of $T$, or $e$ is burnt before $f$. If $e$ is burnt before $f$, then $e$ is the first difference in the labeled edge lists of $T$ and $T'$. Moreover, here $e$ is directed towards $v_0$ in $T$ by our assumption. Hence indeed, $T'\prec T$.
	If $f$ is burnt before $e$, then $f$ is the first difference in the labeled edge lists of $T$ and $T'$. Note that $C^*(T',f)=C^*(T,e)$. Moreover, since $e$ and $f$ stand opposite in the fundamental cut $C^*(T,e)$, $f$ is directed away from $v_0$ in $T'$. This implies $T'\prec T$ in this case, too.
	
	Next, we show that for any $B_1, B_2 \in \mathcal{D}$, $\Delta_{B_1}$ and $\Delta_{B_2}$ are interior-disjoint. For this, it is enough to show a hyperplane that separates them.
	
	Let $F_1 = E-B_1$ and $F_2 = E-B_2$ be the two spanning arborescences corresponding to $B_1$ and $B_2$. Take the first place where their labeled edge lists differ. As they are both spanning arborescences, all their edges are directed away from $v_0$. Hence the first difference needs to be that an edge $e$ is included into (say) $F_1$, $e$ is directed away from $v_0$ in $F_1$ and $e$ is not included into $F_2$. 
	
	We claim that it is enough to show that for the fundamental cycle $C(F_2, e)$, the edges of $C(F_2,e) - F_1$ all stand opposite to $e$ in the cycle $C(F_2,e)$. Indeed, the cycle $C(F_2,e)$ is a cocircuit of the dual matroid $M^*$, which geometrically is a hyperplane $H$ containing all vectors $a_i$ for $i\in E-C(F_2,e)$, moreover, the vectors $a_i$ for  $i\in (C(F_2,e))^+$ are on one side of the hyperplane, and the vectors $a_i$ for $i\in (C(F_2,e))^-$ are on the other side of the hyperplane. 
	$\Delta_{B_2}$ is the convex hull of vectors corresponding to $B_2=E-F_2$. As $e$ is the only edge in $B_2 \cap C(F_2,e)$, all the vertices of $\Delta_{B_2}$ are in $H$ except for the one corresponding to $e$. Hence to conclude that the interiors of $\Delta_{B_1}$ and $\Delta_{B_2}$ are disjoint, it is enough to show that each points corresponding to $B_1$ are either in $H$, or on the opposite side of $H$ compared to $e$. In graph language, that means that each edge of $B_2$ is either in $E-C(F_2,e)$, or if it is in $C(F_2,e)$, then it stands opposite to $e$ in the cycle $C(F_2,e)$. That is, the edges of $C(F_2,e) - F_1$ all stand opposite to $e$ in the cycle $C(F_2,e)$.
	
	Now let us prove that indeed the edges of $C(F_2,e) - F_1$ all stand opposite to $e$ in the cycle $C(F_2,e)$. The cycle $C(F_2,e)$ has the following structure: There is a vertex $u\in C(F_2,e)$ that is closest to $v_0$ within $F_2$. Let $e = \overrightarrow{v_1v_2}$. Then $C(F_2,e)$ consists of two directed paths from $u$ to $v_2$, and one of them contains $e$ as its last edge. As the edge list of $F_2$ is built up so that new edges are always incident to the subtree of $F_2$ built up so far, the edges of the path $uv_1$ are already on the list of $F_2$ before $e$ is added. As the edge lists of $F_1$ and $F_2$ agree until $e$, these edges are also in $F_1$. Hence $C(F_2,e)-F_1$ can only contain edges from the other path of $C(F_2,e)$, and those edges all stand opposite to $e$.
\end{proof}

Now we turn to proving Lemma \ref{lem:arborescences_meet_in_common_face}.
We will need a sufficient condition for two simplices within $\mathcal{Q}_A$ to meet in a proper face. The following lemma is a straightforward generalization of the sufficiency part of Postnikov's lemma \cite[Lemma 12.6]{alex}, that concerns root polytopes of bipartite graphs.

\begin{lemma}\label{lem:triang_cond}
	Let $B_1$ and $B_2$ be two bases of the co-Eulerian regular oriented matroid $M$. If there is no circuit $C=C^+\sqcup C^-$ such that 
	$C^+\subseteq B_1$ and $C^-\subseteq B_2$, then $\Delta_{B_1}$ and $\Delta_{B_2}$ meet in a common face.
\end{lemma}
\begin{proof}
	Suppose that $\Delta_{B_1}$ and $\Delta_{B_2}$ does not meet in a common face. Then there is a point $p\in \Delta_{B_1} \cap \Delta_{B_2}$ such that the minimal face $\Delta_{S_1}$ of $\Delta_{B_1}$ containing $p$ and the minimal face $\Delta_{S_2}$ of $\Delta_{B_2}$ containing $p$ does not agree. Hence we can take $p=\sum_{i\in S_1} \lambda_i a_i$ with $0< \lambda_i$, $\sum_{i\in S_1}\lambda_i=1$ and $p = \sum_{i\in S_2} \mu_i a_i$ with $0< \mu_i$, $\sum_{i\in S_2}\mu_i=1$. Subtracting the second equality from the first, we obtain a linear dependence $\sum_{i\in S_1 \cup S_2} (\lambda_i-\mu_i) a_i = 0$ (where we use the convention $\lambda_i=0$ for $i\notin S_1$ and $\mu_i=0$ for $i\notin S_2$). Clearly, $\lambda_i-\mu_i > 0$ implies $i\in S_1$ and $\lambda_i-\mu_i < 0$ implies $i\in S_2$, moreover, since $S_1\neq S_2$, we do have some nonzero coefficients. By Claim \ref{cl:exist_subcircuit_with_good_orientations} in this case there is a circuit $C$ with $C^+\subseteq \{i : \lambda_i-\mu_i > 0\}\subseteq S_1\subseteq B_1$ and $C^-\subseteq \{i : \lambda_i-\mu_i < 0\}\subseteq S_2\subseteq B_2$. 
\end{proof}
\begin{proof}[Proof of Lemma \ref{lem:arborescences_meet_in_common_face}]
	Let $F_1$ and $F_2$ be two spanning arborescences rooted at $v_0$. By Lemma \ref{lem:triang_cond}, it is enough to prove that there is no signed circuit $C$ in the dual $M^*$ such that $C^+\subseteq E-F_1$ and $C^-\subseteq E-F_2$.
	
	As circuits of the dual are cuts of $G$, we need to prove that there is no cut $C^*$ of $G$ such that $(C^*)^+\subseteq E-F_1$ and $(C^*)^-\subseteq E-F_2$.
	
	Let $C^*$ be any cut of $G$. Any spanning arborescence contains at least one edge from $C^*$ that points from the shore containing $v_0$ to the other shore. Hence for any cut $C^*$, either $(C^*)^+$ or $(C^*)^-$ intersects both $F_1$ and $F_2$. Hence we indeed cannot have a bad cut.
\end{proof}

\begin{proof}[Proof of Theorem \ref{thm:arborescenced_triangulate}]
	The Theorem follows directly from Lemmas \ref{lem:arborescences_dissect} and \ref{lem:arborescences_meet_in_common_face}.
\end{proof}
\begin{coroll}
	The triangulation of the root polytope of $M^*$ by the complements of spanning arborescences rooted at $v_0$ is a geometric embedding of the dual complex of the directed branching greedoid of $G$ rooted at $v_0$.
\end{coroll}

Now we are ready to give our first proof for Theorem \ref{thm:greedoid_poly_as_h^*}
\begin{proof}[First proof of Theorem \ref{thm:greedoid_poly_as_h^*}]
	$\lambda_{G,v_0}$ is defined as the $h$-polynomial of the (abstract) dual complex of the greedoid. By Theorem \ref{thm:arborescenced_triangulate}, the dual complex can be realized geometrically, as a unimodular triangulation of the root polytope $\mathcal{Q}_A$ where $A$ is a totally unimodular matrix representing the (oriented) dual of $G$. By Proposition \ref{prop:matroid_root_polytope_dimension}, $\mathcal{Q}_A$ has dimension $|V|-2$.
    
    Hence by (\ref{eq:h_and_h^*}), the $h$-polynomial of the abstract simplicial complex can be written as $$\lambda_{G,v_0}\left(t\right) = t^{|V|-1} h^*_{\mathcal{Q}_A}\left(t^{-1}\right).$$
    
\end{proof}

As $\mathcal{Q}_A$ is independent of $v_0$, we obtain the following corollary, first proved by Swee Hong Chan \cite{SweeHong_parking}, using the abelian sandpile model.
\begin{coroll}
	The greedoid polynomial of the directed branching greedoid of an Eulerian digraph is independent of the root vertex.
\end{coroll}

We can deduce another corollary.
\begin{thm}
	For an Eulerian digraph $G$, let $\overleftarrow{G}$ be the (Eulerian) digraph obtained by reversing the orientation of each edge. Let $v_0$ be an arbitrary vertex. Then $\lambda_{G,v_0} = \lambda_{\overleftarrow{G},v_0}$. That is, the greedoid polynomial of an Eulerian directed branching greedoid does not change if each edge is reversed.
\end{thm}
\begin{proof}
    Notice that the oriented matroid structures of $G$ and $\overleftarrow{G}$ agree. Hence, their dual oriented matroids also agree. Now the statement follows from Theorem \ref{thm:greedoid_poly_as_h^*}.
\end{proof}

\section{A formula for the $h^*$-polynomial of the root polytope of a co-Eulerian matroid, using activities}

In this section, we give an alternative proof for Theorem \ref{thm:greedoid_poly_as_h^*}, using the definition of the greedoid polynomial via activities. 
Apart from the root-independence of the greedoid polynomial, this proof also yields further alternative definitions for the greedoid polynomial.

We give a formula for the $h^*$-polynomial of $\mathcal{Q}_A$ for a co-Eulerian matroid. This formula is the analogue of \cite[Theorem 1.4]{fixed_edge_order} that considered the special case of root polytopes of semi-balanced digraphs (that are exactly the graphical co-Eulerian matroids).

For a basis $B$ of a matroid, upon adding an element $e\notin B$ to $B$, there is exactly one circuit that is a subset of $B\cup e$. This (signed) circuit is called the fundamental circuit of $e$ with respect to $B$ and it is denoted by $C(B,e)$.

For a basis $B$ of a matroid, upon removing an element $e\in B$ from $B$, there is exactly one cocircuit that is a subset of the complement of $B-e$. This (signed) cocircuit is called the fundamental cocircuit of $e$ with respect to $B$ and it is denoted by $C^*(B,e)$.

We will need the following activity notions, first appearing in \cite{LiPostnikov,SweeHong_parking,hyperBernardi}. Note that external semi-activity for oriented matroids is a generalization of external-semi-activity for digraphs.

\begin{defn}[internal semi-activity in oriented matroids]
	Let $M$ be an oriented matroid with a fixed ordering of the elements. Let $B$ be a basis of $M$. An element $e\in B$ is internally semi-active for $B$ if in the fundamental cocircuit $C^*(B,e)$, the maximal element (with respect to the fixed ordering) has the same sign as $e$. If the maximal element has opposite sign as $e$, then we say that $e$ is internally semi-passive for $B$.
\end{defn}
\begin{defn}[external semi-activity in oriented matroids]
	Let $M$ be an oriented matroid with a fixed ordering of the elements. Let $B$ be a basis of $M$. An element $e\notin B$ is externally semi-active for $B$ if in the fundamental circuit $C(B,e)$, the maximal element (with respect to the fixed ordering) has the same sign as $e$. If the maximal element has opposite sign as $e$, then we say that $e$ is externally semi-passive for $B$.
\end{defn}

Let us first note the following connection, that follows from Fact \ref{fact:bases_circuits_of_dual}. 

\begin{claim}\label{cl:ext_int_semipassivity_dual}
	Let $M$ be an oriented matroid on the groundset $E$ and $M^*$ its directed dual. Fix an ordering on the ground set $E$. Let $B$ be a basis of $M$. Then an element $e\in B$ is internally semi-passive in $B$ if and only if $e$ is externally semi-passive in $E-B$. \hfill\qedsymbol
\end{claim}

\begin{thm}\label{thm:matroid_interior_poly_fixed_edge_order}
	Let $M$ be a co-Eulerian regular oriented matroid with representing matrix $A$.
	Let us denote the $h^*$-vector of $\mathcal{Q}_A$ by $h^*_A$. Let $\mathcal{D}$ be any dissecting set of bases for $M$, and take any fixed ordering of the elements.
	Then
	$$(h^*_A)_i = |\{B\in \mathcal{D}\mid B \text{ has exactly $i$ internally semi-passive elements}\}|.$$
\end{thm}

\begin{coroll}
Let $M$ be a co-Eulerian regular oriented matroid. For any dissecting set of bases of $M$ and any fixed order on the groundset, the distribution of internal semi-passivities is the same.
\end{coroll}

\begin{proof}[Proof of Theorem \ref{thm:matroid_interior_poly_fixed_edge_order}]
	The proof follows the proof of the special case \cite[Theorem 1.4]{fixed_edge_order} basically word-by-word.
	
	Let us fix a dissecting set of bases $\mathcal{D}$, and an ordering $<$ of the elements of $M$. By renumbering the columns of $A$, we can suppose that the elements of $M$ are ordered as $a_1 < a_2 < \dots < a_m$.
	By Proposition \ref{prop:h^*-vector_from_visible_facets}, to prove the theorem, it is enough to find a point ${q}\in Q_{A}$ in general position, such that for each simplex $\Delta_B$ of the dissection, the number of facets of $\Delta_B$ visible from $q$ is equal to the internal semi-passivity of the basis $B$.
	
	Let $q\in \mathbb{R}^n$ be the following point:
	
	$${q}= \sum_{i=1}^m \left(\frac{t^{i}}{\sum_{j=1}^m t^{j}}\right) {a}_i,$$
	where $t$ is sufficiently large. For us, $t=2$ is enough.
	
	Then ${q}$ is by definition a convex combination of points  ${a}_i$, hence by definition, $q\in \mathcal{Q}_A$.
	
	Next, we show that for any simplex $\Delta_B$, the number of facets visible from $q$ is equal to the internal semi-passivity of the basis $B$.
	
	Let $B\in \mathcal{D}$ be a basis. The facets of $\Delta_B$ are of the form $\Delta_{B-k}=\Conv\{ {a}_i : i \in B - k\}$ for elements $k\in B$.
	We show that $q$ is not in the hyperplane of the facet $\Delta_{B-k}$, furthermore, the facet $\Delta_{B-k}$ is visible from $q$ if any only if $k$ is internally semi-passive in $B$. 

    Take the fundamental cocircuit $C^*(B,k)$. By definition, there exists a linear functional $h$ such that $h(a_i)=0$ for $i\notin C^*(B,k)$, $h(a_i) = 1$ for $i\in (C^*(B,k))^+$ and $h(a_i) = -1$ for $i\in (C^*(B,k))^-$. In particular, $h(a_i)=0$ for $i\in B-k$. Suppose without loss of generality that $k\in (C^*(B,k))^+$, thus, $h(a_k)=1$.
		
	  Hence, $h(p)\geq 0$ for each vertex of $\Delta_B$, and hence $h({p})\geq 0$ for each point
	$p\in \Delta_B$. Next, we show that the facet $\Delta_{B-k}$ is visible from ${q}$ if and only if $h({q}) < 0$.

    Indeed, If $h(q) < 0$, then for any $p\in \Delta_{B-k}$, any point $y$ of $(p,q)$ has $h(y) < 0$, hence $(p,q)$ is disjoint from $\Delta_B$.
    
    If $h(q) \geq 0$, we show a point $p\in \Delta_{B-k}$ such that $(p,q)$ is not disjoint from $\Delta_B$. As $q$ is spanned by $\{a_1, \dots, a_m\}$, and $\{a_i \mid i\in B\}$ is a maximal independent set, we can write $q=\sum_{i\in B} \lambda_i a_i$. Moreover, since $h(q) > 0$, $h(a_k)=1$ and $h(a_i)=0$ for each $i\in B-k$, we have $\lambda_k > 0$.
    Hence if we take $p=\sum_{i\in B-k} \frac{1}{|B-k|} a_i$, then $(1-\varepsilon)p+\varepsilon q \in \Delta_B$ for small enough $\varepsilon$.

    With this, we have shown that $\Delta_{B-k}$ is visible from ${q}$ if and only if $h({q}) < 0$. It remains to show that $h(q) < 0$ if and only if $k$ is internally semi-passive in $B$.
    
	By linearity, we have 
	$$h(q)=\sum_{i=1}^m \left(\frac{t^{i}}{\sum_{j=1}^m t^{j}}\right) h(a_i).$$
	As $h({a}_i)=0$ for $i \notin C^*(B,k)$, these edges do not contribute to $h(q)$. The edges of $C^*(B,k)$ having the same sign as $k$ have $h(a_i) = 1$ and the edges of $C^*(B,k)$ having opposite sign as $k$ have $h({a}_i) = -1$. Hence $h(q)$
	is negative if and only if the largest edge of $C^*(B,k)$ according to $<$ has opposite sign to $k$, that is, if and only if $k$ is internally semi-passive in $B$.
\end{proof}

\begin{proof}[Second proof of Theorem \ref{thm:greedoid_poly_as_h^*}]
	By Lemma \ref{lem:arborescences_dissect}, complements of spanning arborescences give a dissecting set of bases for the dual matroid $M^*$.
	Hence by Theorem \ref{thm:matroid_interior_poly_fixed_edge_order}, 
	$(h^*_{\mathcal{Q}_A})_i$ equals the number of spanning arborescences $F$ rooted at $v_0$ such that the complement $B=E-F$ has internal semi-passivity $i$. By Claim \ref{cl:ext_int_semipassivity_dual}, if $B=E-F$ has internal semi-passivity $i$ in $M^*$ then the arborescence $F$ has external semi-passivity $i$ in $G$. Hence indeed, $h^*_{\mathcal{Q}_A}(t)=\sum_{F\in Arb(G,v_0)} t^{\bar{e}(F)}$. Using the definition of the greedoid polynomial, we get the formula in the statement of the theorem.
\end{proof}

We note that Theorems \ref{thm:matroid_interior_poly_fixed_edge_order} and \ref{thm:greedoid_poly_as_h^*} give many exotic ways to compute the greedoid polynomial $\lambda_{G,v_0}$. Indeed, based on these theorems, the greedoid polynomial can be computed as the generating function of external semi-passivities for any set of spanning trees of $G$ such that the complements of these spanning trees form a dissection of the dual root polytope.

\section*{Acknowledgement}
I would like to thank Swee Hong Chan for suggesting to address the case of $\overleftarrow{G}$, and for helpful suggestions about the introduction. I would like to thank Tamás Kálmán for inspiring discussions. Finally, I would like to thank an anonymous referee for many improving suggestions.

This work was supported by the National Research, Development and Innovation Office of Hungary -- NKFIH, grant no.\ 132488, by the János Bolyai Research Scholarship of the Hungarian Academy of Sciences, and by the ÚNKP-21-5 New National Excellence Program of the Ministry for Innovation and Technology, Hungary. This work was also partially supported by the Counting in Sparse Graphs Lendület Research Group of the Alfr\'ed Rényi Institute of Mathematics.

\bibliographystyle{plain}
\bibliography{Bernardi}

\end{document}